\documentclass[a4paper,draft,reqno,12pt]{amsart}
\usepackage[english]{babel}
\usepackage{amsmath}
\usepackage{amssymb}
\usepackage{amscd}
\usepackage{amsthm}
\usepackage{caption}
\usepackage{euscript}
\usepackage{float}
\usepackage{multirow}
\usepackage{tikz}
\usepackage{url}
\usepackage{pst-node}
\usepackage{tikz-cd}
\usepackage{mathabx}
\usepackage{cellspace} 
\newtheorem{proposition}{Proposition}
\newtheorem{lemma}{Lemma}
\newtheorem{theorem}{Theorem}

\newtheorem{corollary}{Corollary}
\theoremstyle{definition}
\newtheorem{definition}{Definition}
\theoremstyle{definition}
\newtheorem{construction}{Construction}
\newtheorem{example}{Example}

\theoremstyle{remark}
\newtheorem {remark}{Remark}

\DeclareMathOperator{\Spec}{Spec}
\DeclareMathOperator{\homo}{Hom}
\DeclareMathOperator{\Aut}{Aut}

\DeclareMathOperator{\Pic}{Pic}

\DeclareMathOperator{\GL}{GL}

\DeclareMathOperator{\supp}{Supp}

\def\Im{{\rm Im}\,}

\def\BG{{\mathbb G}}
\def\BK{{\mathbb K}}

\def\BG{{\mathbb G}}
\def\BF{{\mathbb F}}

\def\BK{{\mathbb K}}

\def\BZ{{\mathbb Z}}

\def\BQ{{\mathbb Q}}

\def\BP{{\mathbb P}}
\def\BA{{\mathbb A}}

\def\CO{\mathcal{O}}

\def\CL{\mathcal{L}}

\def\Soc{\mathrm{Soc}}
\def\Cl{\mathrm{Cl}}
\def\Pic{\mathrm{Pic}}

\def\div{\mathrm{div}}

\def\mf{\mathfrak{m}}
\def\Ann{\mathrm{Ann}}
\def\supp{\mathrm{supp}}

\title{Monomial algebras and $\mathbb{G}_a^n$-equivariant embeddings into toric varieties}

\author{Alexander Chernov}
\address{%
    \begin{minipage}{\textwidth}
        \setlength{\parindent}{0pt}
        \textbf{Email:} \url{aleksander.chernov@math.msu.ru}\\[4pt]
        Lomonosov Moscow State University, Faculty of Mechanics and Mathematics,\\
        Department of Higher Algebra, Leninskie Gory 1, Moscow, 119991 Russia;\\[2pt]
        and\\[2pt]
        HSE University, Faculty of Computer Science,\\
        Pokrovsky Boulevard 11, Moscow, 109028, Russia
    \end{minipage}%
}
\thanks{\endgraf This article is an output of a research project (HSE-BR-2025-22) implemented as part of the Basic Research Program at HSE University. \endgraf \emph{Key words and phrases.} Toric variety, additive action, locally nilpotent derivation, automorphism, unipotent group.}
\begin{document}
\maketitle

\begin{abstract}
An induced additive action on a projective variety $X\subseteq\mathbb{P}^n$ is a regular action of the group $\mathbb{G}_a^m$ on $X$ with an open orbit that can be extended to a regular action on~$\mathbb{P}^n$. Such actions are known to correspond to pairs $(A, U)$, where $A$ is a local algebra and~$U$ is a generating subspace lying in the maximal ideal. This paper studies additive actions on projective toric varieties, with a particular focus on toric surfaces. We prove that for any linearly normal toric variety equipped with a torus-normalized additive action, the associated pair consists of a monomial algebra and a subspace spanned by variables. Also we describe pairs that correspond to additive actions on toric surfaces in low-dimensional projective spaces.
\end{abstract}

\section{Introduction}
Let $\BK$ be an algebraically closed field of characteristic zero. By a variety or an algebraic group we always mean an algebraic variety or an algebraic group over~$\BK$. By open and closed subsets of algebraic varieties we always mean open and closed in Zariski topology subsets. By \emph{local algebra} we mean a commutative associative algebra over~$\BK$ with a unit and a unique maximal ideal. We denote by~$\BG_a$ the algebraic group~$(\BK, +)$ and by $\BG_m$ the group~$(\BK^\times,\cdot);$ by \emph{torus} we mean the algebraic group~$\BG_m^k=\BG_m\times\dotsc\times\BG_m$ ($k$-times). We say that an algebraic variety is \emph{toric} if it is normal and admits an algebraic torus action with an open orbit.

\begin{definition}
    An \emph{additive action} on an algebraic variety $X $ is a regular effective action of $\BG_a^m$ on $X$ with an open orbit. By \emph{induced additive action} on a projective algebraic variety $X \subseteq \BP^n$ we mean a regular effective action of $\BG_a^m$ on $\BP^n$ such that $X$ is the closure of a $\BG_a^m$-orbit.   
\end{definition}

Not all additive actions on projective varieties are induced (see \cite[Example 1]{AP}). But if a projective variety $X\subseteq \BP^n$ is normal and linearly normal, i.e. the restriction map~$H^0(\BP^n,\CO_{\BP^n}(1))\to H^0(X, \CO_{\BP^n}(1)\big|_X)$ is surjective, then any additive action on $X$ is induced.

Knop and Lange studied regular actions of commutative algebraic groups on projective space with an open orbit in \cite{KL}. Hassett and Tschinkel focused on equivariant compactifications of~$\BG_a^n$ in~\cite{HT}. In both works authors found a one-to-one correspondence between local algebras of dimension $n+1$ and additive actions on $\BP^n$. Further, this theory has been extensively developed in many directions. For instance, the case of additive actions on projective hypersurfaces was studied in~\cite{ABZ,AP,Be,BCS,Sha}. Results in~\cite{BM, UD, UL, ZM} are related to additive actions on Fano varieties.

In our paper, we study the case of toric varieties. Remarkable results about additive actions on toric varieties were found in~\cite{AR}. Authors proved that a complete toric variety admits an additive action if and only if it admits a \emph{normalized} additive action, i.e. an acting torus normalizes $\BG_a^n$ in the group of automorphisms. Moreover, they proved that any two normalized additive actions are equivalent; see \cite[Theorem 2]{AR} and \cite[Theorem 3]{AR}.

The key definition in the theory of additive actions on projective hypersurfaces is $H$-pair; see~\cite[Definition 2.12]{AZ}. We introduce the following more general definition for the case of an arbitrary subvariety in the projective space.
\begin{definition}
    A pair $(A, U)$ is called an~\emph{$S$-pair} if~$A$ is a local algebra with a maximal ideal~$\mf$ and~$U\subseteq\mf$ is a subspace generating~$A$ as a~$\BK$-algebra.
\end{definition}

There is a one-to-one correspondence between induced additive actions on projective subvarieties of dimension $m$ in $\BP^n$ and $S$-pairs $(A, U)$, where $A$ is a local algebra of dimension~$n+1$ and $U$ is a subspace of dimension $m$ lying in the maximal ideal of $A$; see~\cite[Theorem 2.6]{AZ}.

In Section 3, we investigate the correspondence between $S$-pairs and induced additive actions on projective toric varieties. It turns out that for linearly normal projective toric varieties, the normalized additive action corresponds to the $S$-pair consisting of a monomial algebra (Definition~\ref{monomialdefinition}) and a subspace spanned by variables~(Theorem~\ref{algebranorm}). Conversely, any such $S$-pair yields a variety admitting a torus action with a normalized induced additive action; however, the variety might not be normal (Proposition~\ref{monomialhomo}, Example~\ref{nonnormalinduced}). 

Section~\ref{lowdimalgebras} focuses on the case of toric surfaces. Dzhunusov proves in~\cite{DZ} that a complete toric surface admits at most two non-equivalent additive actions. We describe explicitly $S$-pairs corresponding to linearly normal toric surfaces equipped with an additive action in a projective space of dimension at most 5 (Table~1 and Table~2). It turns out that such surfaces are either weighted projective planes or Hirzebruch surfaces.

\section{The Hassett-Tschinkel correspondence}

A projective variety $X \subseteq \BP^n$ is called \emph{non-degenerate} if it is not contained in any hyperplane of $\BP^n$. The following theorem is central to our topic.
\begin{theorem}\cite[Theorem 2.6]{AZ}\label{correspondence}
There is a one-to-one correspondence between
\begin{enumerate}
    \item equivalence classes of induced additive actions on non-degenerate subvarieties in $\BP^n$; 
    \item isomorphism classes of pairs $(A, U)$, where $A$ is a local $(n+1)$-dimensional algebra with a maximal ideal $\mathfrak{m}$ and $U$ is a subspace in $\mathfrak{m}$ that generates $A$ as an algebra with unit. 
\end{enumerate}
\end{theorem}

In Theorem~\ref{correspondence} two $S$-pairs $(A_1, U_1)$ and $(A_2, U_2)$ are isomorphic if there is an isomorphism $\varphi:A_1\to A_2$ such that $\varphi(U_1)=U_2$.

We describe briefly the correspondence in Theorem~\ref{correspondence}. Consider a local algebra~$A~=~\BK~\oplus~\mf$ of dimension $n+1$ and a subspace $U$ of dimension $m$ generating~$A$ as a~$\BK$-algebra. There is an isomorphism $\exp:(\mf, +) \to (1 +\mf, \cdot)$ of the standard form: \begin{equation*}
    \exp(x) = \sum_{k=0}^\infty\frac{x^k}{k!}.
\end{equation*} This formula is well-defined since any element from $\mf$ is nilpotent and hence the series is finite. Restricting the exponent homomorphism to the subspace $U$, we get an isomorphism $\exp: (U, +) \to (\exp U, \cdot)$. Consequently, the group~$(U,+)\simeq\BG_a^m$ acts on ~$A$ as~$(\exp U,\cdot)$ by the standard multiplication inducing the action on~$\BP(A)$. It induces an action of~$\BG_a^m$ on~$\BP(A)\simeq\BP^n$. Denote by $\pi:A\backslash\{0\}\to\BP(A)$ the standard projection. Then the closure \begin{equation*}
    X = \overline{\pi(\exp(U)\cdot1)}
\end{equation*} is the corresponding subvariety in $\BP(A)$ with an induced additive action.

Conversely, let~$X \subseteq \BP^n=\BP(V)$ be a subvariety with an induced additive action of~$\BG_a^m$. One can show that every~$\BG_a^m$-action on~$\BP(V)$ lifts to a linear action on~$V$, so we have the corresponding inclusion~$\varphi:\BG_a^m \to\GL_{n+1}(\BK) \simeq \GL(V)$. Let $(g_1, \dotsc, g_m)$ be coordinates on~$\BG_a^m$. Consider the matrices\begin{equation*}
    T_i = \frac{\partial \varphi(g_1,\dotsc,g_m)}{\partial g_i}\Bigg|_{(g_1,\dotsc,g_m)=(0,\dotsc,0)} \in \text{M}_{n+1}(\BK), \ i = 1,\dotsc, m.
\end{equation*} Then the matrices $T_1, \dotsc, T_m$, as elements of the associative algebra $\text{M}_{n+1}(\BK)$, generate a local algebra $A\subseteq \text{M}_{n+1}(\BK)$ of dimension $n+1$ and form a generating subspace in the maximal ideal of $A$. One can note that the subspace spanned by matrices $T_i$ is just the image of the differential of the map $\varphi$ at unity.

Now we describe the construction of the corresponding $S$-pair for normal varieties in more geometric terms. As we know, if $X$ is a normal algebraic variety, its divisor class group~$\Cl(X)$ contains the Picard group~$\Pic(X)$. So for any invertible sheaf~$\CL\in~\Pic(X)$ there is an isomorphism~$\CL\simeq\CO_X(D)$ for some divisor~$D$ on~$X$. For an invertible sheaf~$\CO_X(D)$, we use the following representation of its global sections: \begin{equation}
    H^0(X, \CO_X(D)) = \{f\in K(X) \ \big| \ \div(f) + D \geq 0\} \cup \{0\}.
\end{equation}

It is well known that for complete varieties it is a finite-dimensional vector space over~$\BK$.

\begin{proposition}\cite[Proposition 3.1]{AZ}\label{picaff}
    Let $X$ be a normal variety and $U$ be an open subset in $X$ isomorphic to the affine space. Then the divisor class group $\Cl(X)$ is freely generated by the prime divisors lying in the complement of the open subset $U$.
\end{proposition}

If a variety $X$ admits an additive action by $\BG_a^m$, then the divisors in the complement of the open orbit are $\BG_a^m$-invariant and generate freely $\Cl(X)$.

There is a very useful algebraic tool in the theory of $\BG_a$-actions on affine varieties.

\begin{definition}
     The derivation $\delta$ on a commutative algebra $B$ is called \emph{locally nilpotent} if for every element~$b\in B$ there is a natural number $n = n(b)$ such that $\delta^n(b) = 0$.
\end{definition}

It is well known that every $\BG_a$-action on an affine variety $X$ is induced by some locally nilpotent derivation $\delta$ on its algebra of regular functions~$\BK[X]$ (see~\cite[Section 1.5.1]{FR}). For every $t\in\BK$ the mapping $\exp{(t\delta)}$ is an automorphism of $\BK[X]$. Since $\exp((t+s)\delta) = \exp{(t\delta)}\exp{(s\delta)}$, it defines an action of $\BG_a$ on $\BK[X]$ and thus it defines a $\BG_a$-action on $X$.

Below we describe the construction of an $S$-pair that corresponds to an induced additive action on a normal projective variety in terms of derivations on its algebra of regular functions on the open $\BG_a^m$-orbit.

\begin{construction}\label{dersalgebra}
    Suppose $X \subseteq \BP^n$ is a normal subvariety with an induced additive action of~$\BG_a^m$. Then this closed embedding is induced by a subspace $V\subseteq H^0(X, \CO_X(D))$ for some very ample divisor $D$. We can assume that~$D$ is~$\BG_a^m$-invariant by Proposition~\ref{picaff}. Fix the coordinates~$(x_1,\dotsc,x_m)$ on the open $\BG_a^m$-orbit. Since the $\BG_a^m$-action on~$X$ is induced, it extends to a linear action on $H^0(X, \CO_X(D))$. One can check that under this extension,~$V$ is an ($n+1$)-dimensional~$\BG_a^m$-invariant subspace. Next, note that the vector space~$V$ consists of functions which have the form $f(x_1,\dotsc,x_m)$ on the open orbit, where $f$ is a polynomial. The~$\BG_a^m$-action on the open orbit is induced by $m$ commuting locally nilpotent derivations~$d_1,\dotsc,d_m$ on~$\BK[x_1,\dotsc,x_m]$. The subspace~$V$ is $\BG_a^m$-invariant and hence it is invariant under the derivations~$d_1,\dotsc,d_m$. The corresponding $S$-pair~$(A,U)$ has the following form: \begin{equation}
        A = \BK[d_1\big|_V,\dotsc,d_m\big|_V] \subseteq L(V), \ U = \bigoplus_{i=1}^m\BK d_i\big|_V.
    \end{equation}
\end{construction}

This construction is also described in~\cite[Theorem 2.6]{AZ}

\begin{remark}
    Every additive action determines a choice of coordinates on the open orbit in which it acts by translations, making the derivations $d_1,\dotsc,d_m$ simply $\partial/\partial x_1,\dotsc,\partial/\partial x_m$. But these coordinates might be inconvenient to work with. For example, the case of toric varieties is exactly like this.
\end{remark}

\section{Induced additive actions on toric varieties}

The following theorem is central to the theory of additive actions on toric varieties.

\begin{theorem}\cite[Theorem 3]{AR}\label{normandany}
    Let $X$ be a complete toric variety with an acting torus~$\BG_m^n$. The following conditions are equivalent:\begin{enumerate}
        \item[1)] X admits an additive action normalized by the acting torus $\BG_m^n$;
        \item[2)] X admits an additive action;
        \item[3)] a maximal unipotent subgroup of the group $\Aut(X)$ acts on $X$ with an open orbit.
    \end{enumerate}
\end{theorem}

As we mentioned in the introduction, any two normalized additive actions on a toric variety are isomorphic. One can see the proof of this fact in~\cite[Theorem 2]{AR}.

\begin{remark}\label{normaction}
    It is easy to check directly that for toric varieties with normalized additive action, the open $\BG_m^n$-orbit lies in the open $\BG_a^n$-orbit. Also, the divisors lying in the complement of the open $\BG_a^n$-orbit are $\BG_m^n$-invariant since the torus acts on them by permutations, but this action is trivial because a torus is a connected group.
\end{remark}

We remind the reader that a group of automorphisms of a complete toric variety is an affine algebraic group; see \cite{DM}.

\begin{proposition}\label{openorbit}
    Let $X$ be a complete toric variety that admits an additive action. Then there exists an equivalent additive action such that its open orbit contains an open torus orbit.
\end{proposition}

\begin{proof}
    Denote by $\widetilde{\BG_a^n} \subseteq \Aut(X)$ a subgroup corresponding to some additive action and by $\BG_a^n \subseteq \Aut(X)$ corresponding to a normalized one. Let $U, \widetilde{U} \subseteq \Aut(X)$ be maximal unipotent subgroups containing $\BG_a^n$ and $\widetilde{\BG_a^n}$ respectively. They act on $X$ with the same open orbits as $\BG_a^n$ and $\widetilde{\BG_a^n}$. It follows from the fact that orbits of a unipotent groups are isomorphic to an affine space, and an affine space cannot contain another affine space as a proper open subset. Since all maximal unipotent subgroups are conjugate in an affine algebraic group, we can find an automorphism~$\varphi\in\Aut(X)$ such that~$\varphi^{-1}\widetilde{U}\varphi=U$. Consider the following equivalent additive action of $\BG_a^n$ on $X$:\begin{equation*}
        g\cdot x = \varphi^{-1}(g\cdot\varphi(x)).
    \end{equation*} This action has the same open orbit as $\widetilde{\BG_a^n}$ and thus contains an open $\BG_m^n$-orbit.
\end{proof}

For the normalized additive action we can choose coordinates $(x_1,\dotsc,x_n)$ on the open~$\BG_a^n$-orbit in which~$\BG_a^n$ acts by shifts (i.e. the action is induced by~$n$ commuting homogeneous derivations~$\partial/\partial x_1,\dotsc, \partial/\partial x_n$) and coordinates $(t_1,\dotsc,t_n)$ on $\BG_m^n$ such that $\BG_m^n$ acts on the open orbit by coordinate-wise multiplication. One can see the proof in~\cite[Section 4.1]{AZ}. We call such coordinates (on both open orbit and an acting torus) \emph{standard}.

\begin{definition}\label{monomialdefinition}
    A finitely generated $\BK$-algebra $A$ is called \emph{monomial} if $A$ is isomorphic to the quotient of the ring of polynomials in $n$ variables $\BK[x_1,\dotsc,x_n]/I$, where $I$ is a \emph{monomial} ideal, i.e. it is generated by monomials. 
\end{definition}

\begin{proposition}\label{monomialhomo}
    Let~$A=\BK[x_1,\dotsc,x_n]/I$ be a monomial finite-dimensional local algebra and $U=\langle x_1, \dotsc, x_n\rangle$ be the subspace spanned by variables. Then a variety $X$ with an induced additive action corresponding to the~$S$-pair~$(A, U)$ admits a torus action with an open orbit, and the additive action is normalized by the acting torus.
\end{proposition}

\begin{proof}
    Let~$\pi:A\backslash\{0\}\to\BP(A)$ be the canonical projection morphism. By Theorem~\ref{correspondence} we have~$X=\overline{\pi(\exp(U)\cdot1)}\subset\BP(A)$. We have \begin{equation}\label{firstequation}
        \exp(U)\cdot1 = \Bigg\{\sum \frac{1}{i_1!\dotsc i_n!}\alpha_1^{i_1}\dotsc \alpha_n^{i_n}x_1^{i_1}\dotsc x_n^{i_n}, \ \alpha_1,\dotsc,\alpha_n \in \BK\Bigg\}.
    \end{equation}
    Consider the following linear representation of~$\BG_m^n$ on~$A$: \begin{equation}\label{secondequation}
        (t_1, \dotsc, t_n) \circ x_1^{i_1}\dotsc x_n^{i_n} = t_1^{i_1}\dotsc t_n^{i_n}x_1^{i_1}\dotsc x_n^{i_n}, \ (t_1,\dotsc,t_n) \in \BG_m^n.
    \end{equation}
    This action induces an action on~$\BP(A)$ and commutes with the projection~$\pi$. It is not hard to see from~(\ref{firstequation}) and~(\ref{secondequation}) that the closure of the $\BG_m^n$-orbit of the point~$\pi(\sum\frac{1}{i_1!\dotsc i_n!}x_1^{i_1}\dotsc x_n^{i_n})$ coincides with~$X$. This means that~$X$ admits a torus action with an open orbit. This action is normalized because, in coordinates induced by the action of~$\BG_a^n$, the torus acts by coordinate-wise multiplication.
\end{proof}

The following example shows that the corresponding variety might not be normal.

\begin{example}\label{nonnormalinduced}
    Consider the $S$-pair $A = \BK[t,s]/(t^3,ts,s^3)$ and $U = \BK t \oplus \BK s$. We have \begin{equation*}
        \exp(U)\cdot 1 = \Big\{1 + at + bs + \frac{a^2}{2}t^2+\frac{b^2}{2}s^2 \ \Big| \ a,b\in\BK\Big\},
    \end{equation*} so the corresponding variety with an induced additive action is the closure of the set \begin{equation*}
        \Big\{\Big[1:a:b:\frac{a^2}{2}:\frac{b^2}{2}\Big], \ a,b\in\BK\Big\} \subseteq \BP^4
    \end{equation*} in homogeneous coordinates in $\BP^4$. Denote this closure by $X$. Indeed, $X$ is the intersection of two quadrics \begin{gather}\label{quadriceq}
        \begin{cases}
            z_1^2 = 2z_0z_3, \\
            z_2^2 = 2z_0z_4,
        \end{cases}
    \end{gather} because the variety given by these equations and $X$ coincide on the open affine chart and both are irreducible. In the affine chart $z_3\neq0$ in coordinates \begin{equation*}
        x_1 = z_0/z_3, \ x_2 = z_1/z_3, \ x_3 = z_2/z_3, \ x_4 = z_4/z_3,
    \end{equation*} the equations (\ref{quadriceq}) have the following form: \begin{gather*}
        \begin{cases}
            x_2^2 = 2x_1, \\
            x_3^2 = 2x_1x_4,
        \end{cases}
    \end{gather*} and it is easy to see that it has a line of singular points $x_1=x_2=x_3=0$. Thus, the complement of the open $\BG_a^2$-orbit in $X$ consists of singular points. Its codimension equals~1, and hence $X$ is not normal. 
\end{example}

Now we look at the opposite situation. Consider a projective toric variety with a normalized additive action and its $\BG_a^n$-equivariant embedding in projective space. What can we say about the corresponding $S$-pair? We answer this question in the case when the variety is linearly normal.

\begin{proposition}\cite[Proposition 4.3.2]{C}\label{sectionstoric}
    Let $X$ be a toric variety with the action of $\BG_m^n$. If~$D$ is a $\BG_m^n$-invariant Weil divisor on X, then \begin{equation}
        H^0(X, \CO_X(D)) = \bigoplus_{\div(\chi^m)+D\geq0}\BK\cdot\chi^m,
    \end{equation} where $\chi^m = \chi^{(m_1,\dotsc,m_n)}$ are the characters of $\BG_m^n$.
\end{proposition}

\begin{proposition}\label{toricsections}
    Let $X$ be a toric variety with a normalized additive action. Let~$D$ be a divisor on~$X$ and~$\supp(D)$ lie in the complement of the open~$\BG_a^n$-orbit. Choose the standard coordinates on the open orbit and an acting torus. \begin{enumerate}
        \item[1)] If $\chi^{(m_1,\dotsc,m_n)}\in H^0(X, \CO_X(D))$, then $m_i \geq 0$ \ for \ $i=1,\dotsc,n$;
        \item[2)]if $\chi^{(m_1,\dotsc,m_n)}\in H^0(X, \CO_X(D))$, then $\chi^{(k_1,\dotsc,k_n)} \in H^0(X, \CO_X(D))$ for all tuples~$(k_1,\dotsc,k_n)$ with $0\leq k_j\leq m_j$.
    \end{enumerate}
\end{proposition}

\begin{proof}
    1) Every $m_i$ is non-negative since $\chi^m$ is regular on the open $\BG_a^n$-orbit.
    
    2) Denote by $(x_1,\dotsc,x_n)$ the standard coordinates. Since~$D$ lies in the complement of the open~$\BG_a^n$-orbit,~$D$ is~$\BG_a^n$-invariant. This yields a representation of~$\BG_a^n$ in~$H^0(X,\CO_X(D))$. But since~$\BG_a^n$ acts by shifts in coordinates~$(x_1,\dotsc,x_n)$, then~$\partial\chi^m/\partial x_i~\in~H^0(X,\CO_X(D))$ for~$i=1,\dotsc,n$. This finishes the proof.
\end{proof}

Further we consider the case when equivariant embeddings are induced by complete linear systems, so the following proposition will be very useful.

\begin{proposition}\cite[Section 2]{AP}\label{linnormind}
    Let $X\subseteq\BP^n$ be a normal and linearly normal subvariety that admits an additive action. Then this action is induced.
\end{proposition}

Now we are ready to prove the main result of this section. 

\begin{theorem}\label{algebranorm}
    Let $X$ be a complete toric variety with a normalized additive action and~$D$ be a very ample divisor on $X$ lying in the complement of the open $\BG_a^n$-orbit. Consider a linearly normal embedding in~$\BP^N$ with the invertible sheaf~$\CO_X(D)$. Then the additive action is induced and the corresponding $S$-pair~$(A,U)$ has the following form:\begin{gather}
        A =  H^0(X, \CO_X(D)) = \bigoplus_{\div(\chi^m)+D\geq0}\BK\chi^m; \\
        U = \BK\chi^{(1,0,\dotsc,0)}\oplus\BK\chi^{(0,1,\dotsc,0)}\oplus\dotsc\oplus\BK\chi^{(0,0,\dotsc,1)},
    \end{gather} where the multiplication of the characters is the standard multiplication and $\chi^m\cdot\chi^{m'} = 0$ in the case $\chi^{m+m'} \notin A$. In particular, the algebra $A$ is monomial.
\end{theorem}

\begin{proof}
    The action is induced by Proposition~\ref{linnormind}. Since $D$ lies in the complement of the open $\BG_a^n$-orbit, it is invariant under both actions. By Proposition~\ref{toricsections} if~$\chi^{(m_1,\dotsc,m_n)}~\in~H^0(X, \CO_X(D))$, then~$\chi^{(k_1,\dotsc,k_n)}\in A$ for~$0\leq k_j\leq m_j$. This yields the structure of the monomial algebra on~$A$. Fix the torus coordinates~$(t_1,\dotsc,t_n)$ on the open~$\BG_a^n$-orbit. By Construction~\ref{dersalgebra} the corresponding~$S$-pair is $\BK[\partial/\partial t_1\big|_A,\dotsc,\partial/\partial t_n\big|_A]$. The following lemma completes the proof.

    \begin{lemma}
        Let $B = \BK[s_1,\dotsc,s_n]/I$ be a monomial algebra with respect to a monomial ideal~$I$. Then the algebra~$\BK[\partial/\partial s_1\big|_B,\dotsc,\partial/\partial s_n\big|_B]$ is isomorphic to $B$.
    \end{lemma}

    \begin{proof}
        Denote by $\partial_i\coloneq\partial/\partial s_i|_B$. First, we prove that~$\partial_1^{k_1}\dotsc\partial_n^{k_n}= 0$ if and only if~$s_1^{k_1}\dotsc s_n^{k_n} =0$  in $B$. 

        Suppose $\partial_1^{k_1}\dotsc\partial_n^{k_n} = 0$. This immediately yields $s_1^{k_1}\dotsc s_n^{k_n} = 0$, since otherwise\begin{equation*}
            \partial_1^{k_1}\dotsc\partial_n^{k_n}(s_1^{k_1}\dotsc s_n^{k_n})= k_1!\dotsc k_n!\neq0.
        \end{equation*}
        Conversely, if $s_1^{k_1}\dotsc s_n^{k_n} = 0$ in $B$, then, since $B$ is a monomial algebra, there are no nonzero monomials $s_1^{l_1}\dotsc s_n^{l_n}$ in $B$ such that $l_i\geq k_i$. Hence if $s_1^{l_1}\dotsc s_n^{l_n}$ is a nonzero monomial in $B$, then $l_j<k_j$ for some $j$. But this yields~$\partial_j^{k_j}(s_1^{l_1}\dotsc s_n^{l_n})=0$. In particular, $\partial_1^{k_1}\dotsc\partial_n^{k_n}(s_1^{l_1}\dotsc s_n^{l_n})=0$. It means that $\partial_1^{k_1}\dotsc\partial_n^{k_n} = 0$.

        Now, note that all nonzero $\partial_1^{k_1}\dotsc\partial_n^{k_n}$ are linearly independent. Indeed, let \begin{equation*}          \sum_{\partial_1^{i_1}\dotsc\partial_n^{i_n}\neq0}\alpha_{i_1,\dotsc,i_n}\partial_1^{i_1}\dotsc\partial_n^{i_n}=0, \ \alpha_{i_1,\dotsc,i_n}\in\BK,
        \end{equation*}

        and $\alpha_{k_1,\dotsc,k_n}\neq0$ for some tuple $(k_1,\dotsc,k_n)$. Apply the derivation on the left side to the monomial $s_1^{k_1}\dotsc s_n^{k_n}$. We get \begin{equation*}
            \alpha_{k_1,\dotsc,k_n}+M = 0,
        \end{equation*} where $M$ is a linear combination of non-constant monomials and hence $\alpha_{k_1,\dotsc,k_n}=0$. This is a contradiction.

        Finally, consider the map $s_i\mapsto\partial_i$ from $B$ to $\BK[\partial_1,\dotsc,\partial_n]$. It is a well-defined homomorphism of algebras. The above reasoning shows that this map is bijective and hence it yields the required isomorphism.
    \end{proof}
\end{proof}

\begin{example}\label{example1}
    Consider the variety $\BP^1\times\dotsc\times\BP^1$. It is a smooth toric variety that admits a unique, up to equivalence, additive action and this action has to be equivalent to the normalized one by Theorem~\ref{normandany}. The uniqueness follows from the fact that\begin{equation*}
        \Aut(\underbrace{\BP^1\times\dotsc\times\BP^1}_n) \simeq (PGL_2(\BK))^n\rtimes S_n
    \end{equation*} and hence a maximal unipotent group is isomorphic to $\BG_a^n$. Consequently, all $\BG_a^n$-subgroups in the automorphism group are conjugate and induce equivalent additive actions.
    
    Choose the homogeneous coordinates~$[z_0:z_1]$ on $\BP^1$. Then $\BG_m$ and $\BG_a$ act on every factor \begin{equation*}
        t\cdot[z_0:z_1] = [z_0:tz_1], \ a\cdot[z_0:z_1] = [z_0:z_1+az_0], \ t\in\BG_m, \ a\in\BG_a.
    \end{equation*} Denote by $\{\infty\}$ the point $[0:1]$; it is fixed under both actions. The divisors \begin{equation*}
        D_i = \BP^1\times\dotsc\times\overbrace{\{\infty\}}^{i}\times\dotsc\times\BP^1
    \end{equation*} form the basis of $\Cl(\BP^1\times\dotsc\times\BP^1)$, and they are invariant under both actions. It is well known that a divisor~$D=k_1D_1+\dotsc+k_nD_n$ is very ample if and only if~$k_i>0$ for all~$i=1,\dotsc,n$. The corresponding global sections can be easily calculated \begin{equation*}
        H^0(\BP^1\times\dotsc\times\BP^1,\CO(k_1D_1+\dotsc+k_nD_n)) = \bigoplus_{0\leq i_j\leq k_j}\BK\chi^{(i_1,\dotsc,i_n)}.
    \end{equation*} Hence by Theorem~\ref{algebranorm} the linearly normal embeddings of $\BP^1\times\dotsc\times\BP^1$ with the normalized additive action correspond to the $S$-pairs \begin{equation}\label{p1copies}
        A = \bigoplus_{0\leq i_j\leq k_j}\BK\chi^{(i_1,\dotsc,i_n)}, \ U = \BK\chi^{(1,0,\dotsc,0)}\oplus\dotsc\oplus\BK\chi^{(0,\dotsc,0,1)}.
    \end{equation} These monomial algebras are Gorenstein, i.e. the \emph{socle} of the algebra $\Soc (A) \coloneq \Ann(\mf_A)$ is of dimension 1, where $\mf_A$ denotes the maximal ideal of $A$. Conversely, any monomial Gorenstein local algebra has form (\ref{p1copies}). Indeed, in this case a one-dimensional socle is generated by a unique maximal monomial and hence any other non-zero monomial divides it. 
    
    This means that there is a one-to-one correspondence between $S$-pairs consisting of a monomial Gorenstein algebra with a subspace spanned by variables and induced additive actions on linearly normal embeddings of $\BP^1\times\dotsc\times\BP^1$.
\end{example}

The following corollary might seem a little tautological since it reformulates Construction~\ref{dersalgebra}. But it shows that all $S$-pairs of a linearly normal toric variety can be obtained by restricting some derivations on a unique local monomial algebra that corresponds to the normalized action.

\begin{corollary}\label{secondalgebra}
    Under the assumptions of Theorem~\ref{algebranorm} let $\alpha$ be an arbitrary additive action on~$X$ induced by derivations $\delta_1,\dotsc,\delta_n$ in standard coordinates. Then it corresponds to the $S$-pair~$(A_\alpha,U_\alpha)$, where\begin{gather*}
        A_\alpha= \BK[\delta_1,\dotsc,\delta_n]/\ker\pi, \ \pi(\delta_i)\coloneq\delta_i\Big|_A, \\
        U_\alpha = \bigoplus_{i=1}^n\BK\cdot\delta_i\Big|_A.
    \end{gather*}
\end{corollary}
\begin{proof}
    Immediately follows from Construction~\ref{dersalgebra} since $H^0(X,\CO_X(D)) = A$.
\end{proof}

Several examples illustrating this corollary will be provided in Section~4. Note that $\dim_\BK A_\alpha = \dim_\BK A$ since by construction these algebras correspond to induced additive actions on the same variety and in the same projective space.

\section{Local algebras of additive actions on projective toric surfaces for low-dimensional embeddings}\label{lowdimalgebras}

In this section we remind the reader of some basic facts about toric varieties and describe local algebras of linearly normal low-dimensional embeddings of projective toric surfaces admitting an additive action.

Let $T = \{(t_1,\dotsc,t_n): t_i\in\BK^\times\}$ be a torus of dimension $n$. Denote by $M~\coloneq~\homo(T,\BK^\times)~\simeq~\BZ^n$ its character lattice and by $N \coloneq \homo(M, \BZ)$ the dual lattice of one-parametric subgroups. By $M_\BQ$ and $N_\BQ$ we denote the $\BQ$-vector spaces $M\otimes_\BZ\BQ$ and~$N\otimes_\BZ\BQ$, respectively. There is a natural pairing \begin{equation*}
    \langle\cdot,\cdot\rangle:M\times N \to \BZ,
\end{equation*} that can be extended to a pairing \begin{equation*}
    \langle\cdot,\cdot\rangle:M_\BQ\times N_\BQ \to \BQ.
\end{equation*}

Let $X$ be a toric variety with the acting torus $T$. Then $X$ corresponds to a fan~$\Sigma\subseteq N_\BQ$. Fix primitive vectors $p_1,\dotsc,p_k$ of the rays of~$\Sigma$. It follows from~\cite[Theorem 3.4]{AR} that $X$ admits an additive action if and only if $\Sigma$ is \emph{bilateral}, i.e. it has the following form: $n$ primitive vectors $p_1, \dotsc,p_n$ of rays of $\Sigma$ form a basis of $N$ and other primitive vectors $p_{n+1},\dotsc,p_k$ lie in the negative octant with respect to the basis~$p_1,\dotsc,p_n$.

It is well known that on a toric variety any divisor is linearly equivalent to a $T$-invariant divisor. Let $D_1,\dotsc,D_k$ be prime divisors corresponding to rays of $\Sigma$. For any divisor $D=\sum a_iD_i$ we can define the lattice polytope \begin{equation*}
    P_D= \{m\in M_\BQ: \ \langle m,p_i\rangle + a_i \geq0\}.
\end{equation*} If $P_D$ is full-dimensional, then the fan $\Sigma$ is a normal fan of $P_D$. Then we have \begin{equation*}
    H^0(X,\CO_X(D)) = \bigoplus_{\div(\chi^m)+D\geq0}\BK\cdot\chi^m = \bigoplus_{m\in P_D}\BK\cdot\chi^m.
\end{equation*}

Consider the morphism \begin{equation*}
    \varphi_D:T\to\BP^{N}, \ \varphi_D(t)=[\chi^{m_0}(t):\dotsc:\chi^{m_N}(t)], \ m_j\in P_D.
\end{equation*} Then if $D$ is very ample, then $X$ is isomorphic to the closure of the image of~$\varphi_D$.

Now we return to the case of toric varieties with an additive action. If a fan is bilateral then the polytope $P_D$ is inscribed in a rectangle, i.e. there exists a vertex $v_0\in P_D$ such that \begin{enumerate}
    \item the primitive vectors on the edges of $P$ containing $v_0$ form a basis $e_1,\dotsc,e_n$ of the lattice $M$;
    \item for every inequality $\langle p,x\rangle\leq a$ on $P$ that corresponds to a facet of $P$ not passing through~$v_0$ we have $\langle p,e_i\rangle \geq 0$ for all $i=1,\dotsc,n$.
\end{enumerate}

It is proved in~\cite{AR} that for a very ample divisor $D$, $X_\Sigma$ admits an additive action if and only if $P_D$ is inscribed in a rectangle. But one can also obtain it from Proposition 4.

 \begin{figure}[h]
        \begin{minipage}{0.48\textwidth}
         \centering 
         \begin{tikzpicture}[x=0.75pt,y=0.75pt,yscale=-1,xscale=1]

\draw    (179.14,1150.46) -- (180.52,1001.46) ;
\draw [shift={(180.54,999.46)}, rotate = 90.53] [color={rgb, 255:red, 0; green, 0; blue, 0 }  ][line width=0.75]    (10.93,-3.29) .. controls (6.95,-1.4) and (3.31,-0.3) .. (0,0) .. controls (3.31,0.3) and (6.95,1.4) .. (10.93,3.29)   ;
\draw    (179.14,1150.46) -- (359,1150.46) ;
\draw [shift={(361,1150.46)}, rotate = 180] [color={rgb, 255:red, 0; green, 0; blue, 0 }  ][line width=0.75]    (10.93,-3.29) .. controls (6.95,-1.4) and (3.31,-0.3) .. (0,0) .. controls (3.31,0.3) and (6.95,1.4) .. (10.93,3.29)   ;
\draw    (249.09,1067.41) ;
\draw    (179.14,1090.06) -- (235.1,1090.06) ;
\draw    (235.1,1090.06) -- (291.05,1150.46) ;
\draw  [dash pattern={on 0.84pt off 2.51pt}]  (235.1,1090.06) -- (207.12,1090.06) ;
\draw [shift={(207.12,1090.06)}, rotate = 180] [color={rgb, 255:red, 0; green, 0; blue, 0 }  ][fill={rgb, 255:red, 0; green, 0; blue, 0 }  ][line width=0.75]      (0, 0) circle [x radius= 1.34, y radius= 1.34]   ;
\draw [shift={(235.1,1090.06)}, rotate = 180] [color={rgb, 255:red, 0; green, 0; blue, 0 }  ][fill={rgb, 255:red, 0; green, 0; blue, 0 }  ][line width=0.75]      (0, 0) circle [x radius= 1.34, y radius= 1.34]   ;
\draw  [dash pattern={on 0.84pt off 2.51pt}]  (179.14,1090.06) -- (207.12,1090.06) ;
\draw [shift={(207.12,1090.06)}, rotate = 0] [color={rgb, 255:red, 0; green, 0; blue, 0 }  ][fill={rgb, 255:red, 0; green, 0; blue, 0 }  ][line width=0.75]      (0, 0) circle [x radius= 1.34, y radius= 1.34]   ;
\draw [shift={(179.14,1090.06)}, rotate = 0] [color={rgb, 255:red, 0; green, 0; blue, 0 }  ][fill={rgb, 255:red, 0; green, 0; blue, 0 }  ][line width=0.75]      (0, 0) circle [x radius= 1.34, y radius= 1.34]   ;
\draw  [dash pattern={on 0.84pt off 2.51pt}]  (179.14,1150.46) -- (207,1150.46) ;
\draw [shift={(207,1150.46)}, rotate = 0] [color={rgb, 255:red, 0; green, 0; blue, 0 }  ][fill={rgb, 255:red, 0; green, 0; blue, 0 }  ][line width=0.75]      (0, 0) circle [x radius= 1.34, y radius= 1.34]   ;
\draw [shift={(179.14,1150.46)}, rotate = 0] [color={rgb, 255:red, 0; green, 0; blue, 0 }  ][fill={rgb, 255:red, 0; green, 0; blue, 0 }  ][line width=0.75]      (0, 0) circle [x radius= 1.34, y radius= 1.34]   ;
\draw  [dash pattern={on 0.84pt off 2.51pt}]  (235,1150.46) -- (207,1150.46) ;
\draw [shift={(207,1150.46)}, rotate = 180] [color={rgb, 255:red, 0; green, 0; blue, 0 }  ][fill={rgb, 255:red, 0; green, 0; blue, 0 }  ][line width=0.75]      (0, 0) circle [x radius= 1.34, y radius= 1.34]   ;
\draw [shift={(235,1150.46)}, rotate = 180] [color={rgb, 255:red, 0; green, 0; blue, 0 }  ][fill={rgb, 255:red, 0; green, 0; blue, 0 }  ][line width=0.75]      (0, 0) circle [x radius= 1.34, y radius= 1.34]   ;
\draw  [dash pattern={on 0.84pt off 2.51pt}]  (263,1150.46) -- (235,1150.46) ;
\draw [shift={(235,1150.46)}, rotate = 180] [color={rgb, 255:red, 0; green, 0; blue, 0 }  ][fill={rgb, 255:red, 0; green, 0; blue, 0 }  ][line width=0.75]      (0, 0) circle [x radius= 1.34, y radius= 1.34]   ;
\draw [shift={(263,1150.46)}, rotate = 180] [color={rgb, 255:red, 0; green, 0; blue, 0 }  ][fill={rgb, 255:red, 0; green, 0; blue, 0 }  ][line width=0.75]      (0, 0) circle [x radius= 1.34, y radius= 1.34]   ;
\draw  [dash pattern={on 0.84pt off 2.51pt}]  (291.05,1150.46) -- (263,1150.46) ;
\draw [shift={(263,1150.46)}, rotate = 180] [color={rgb, 255:red, 0; green, 0; blue, 0 }  ][fill={rgb, 255:red, 0; green, 0; blue, 0 }  ][line width=0.75]      (0, 0) circle [x radius= 1.34, y radius= 1.34]   ;
\draw [shift={(291.05,1150.46)}, rotate = 180] [color={rgb, 255:red, 0; green, 0; blue, 0 }  ][fill={rgb, 255:red, 0; green, 0; blue, 0 }  ][line width=0.75]      (0, 0) circle [x radius= 1.34, y radius= 1.34]   ;
\draw    (207.12,1090.06) -- (206.49,1109.02) ;
\draw [shift={(206.43,1111.02)}, rotate = 271.89] [color={rgb, 255:red, 0; green, 0; blue, 0 }  ][line width=0.75]    (10.93,-3.29) .. controls (6.95,-1.4) and (3.31,-0.3) .. (0,0) .. controls (3.31,0.3) and (6.95,1.4) .. (10.93,3.29)   ;
\draw    (252.12,1108.06) -- (240.77,1120.54) ;
\draw [shift={(239.43,1122.02)}, rotate = 312.28] [color={rgb, 255:red, 0; green, 0; blue, 0 }  ][line width=0.75]    (10.93,-3.29) .. controls (6.95,-1.4) and (3.31,-0.3) .. (0,0) .. controls (3.31,0.3) and (6.95,1.4) .. (10.93,3.29)   ;
\draw    (207,1150.46) -- (207.38,1133.02) ;
\draw [shift={(207.43,1131.02)}, rotate = 91.26] [color={rgb, 255:red, 0; green, 0; blue, 0 }  ][line width=0.75]    (10.93,-3.29) .. controls (6.95,-1.4) and (3.31,-0.3) .. (0,0) .. controls (3.31,0.3) and (6.95,1.4) .. (10.93,3.29)   ;
\draw    (179.26,1121.06) -- (194.43,1121.02) ;
\draw [shift={(196.43,1121.02)}, rotate = 179.86] [color={rgb, 255:red, 0; green, 0; blue, 0 }  ][line width=0.75]    (10.93,-3.29) .. controls (6.95,-1.4) and (3.31,-0.3) .. (0,0) .. controls (3.31,0.3) and (6.95,1.4) .. (10.93,3.29)   ;

\draw (243.27,1072.7) node [anchor=north west][inner sep=0.75pt]  [font=\footnotesize]  {$\chi ^{( 2,2)}$};
\draw (143.16,1120.26) node [anchor=north west][inner sep=0.75pt]  [font=\footnotesize]  {$\chi ^{( 0,1)}$};
\draw (296.16,1127.26) node [anchor=north west][inner sep=0.75pt]  [font=\footnotesize]  {$\chi ^{( 5,0)}$};

\end{tikzpicture}
         \caption{A polytope of a toric variety with an additive action.}
       \end{minipage}
       \begin{minipage}{0.48\textwidth}
         \centering
         \begin{tikzpicture}[x=0.75pt,y=0.75pt,yscale=-1,xscale=1]

\draw    (455.64,1071.13) -- (456.88,1023.02) ;
\draw [shift={(456.93,1021.02)}, rotate = 91.47] [color={rgb, 255:red, 0; green, 0; blue, 0 }  ][line width=0.75]    (10.93,-3.29) .. controls (6.95,-1.4) and (3.31,-0.3) .. (0,0) .. controls (3.31,0.3) and (6.95,1.4) .. (10.93,3.29)   ;
\draw    (455.64,1071.13) -- (506.43,1071.13) ;
\draw [shift={(508.43,1071.13)}, rotate = 180] [color={rgb, 255:red, 0; green, 0; blue, 0 }  ][line width=0.75]    (10.93,-3.29) .. controls (6.95,-1.4) and (3.31,-0.3) .. (0,0) .. controls (3.31,0.3) and (6.95,1.4) .. (10.93,3.29)   ;
\draw    (455.64,1071.13) -- (455.64,1118.02) ;
\draw [shift={(455.64,1120.02)}, rotate = 270] [color={rgb, 255:red, 0; green, 0; blue, 0 }  ][line width=0.75]    (10.93,-3.29) .. controls (6.95,-1.4) and (3.31,-0.3) .. (0,0) .. controls (3.31,0.3) and (6.95,1.4) .. (10.93,3.29)   ;
\draw    (455.64,1071.13) -- (406.86,1118.62) ;
\draw [shift={(405.43,1120.02)}, rotate = 315.77] [color={rgb, 255:red, 0; green, 0; blue, 0 }  ][line width=0.75]    (10.93,-3.29) .. controls (6.95,-1.4) and (3.31,-0.3) .. (0,0) .. controls (3.31,0.3) and (6.95,1.4) .. (10.93,3.29)   ;

\draw (502,1080.93) node [anchor=north west][inner sep=0.75pt]  [font=\scriptsize]  {$( 1,0)$};
\draw (467,1020.93) node [anchor=north west][inner sep=0.75pt]  [font=\scriptsize]  {$( 0,1)$};
\draw (463,1125.93) node [anchor=north west][inner sep=0.75pt]  [font=\scriptsize]  {$( 0,-1)$};
\draw (371,1092.93) node [anchor=north west][inner sep=0.75pt]  [font=\scriptsize]  {$( -1,-1)$};

\end{tikzpicture}
         \caption{The corresponding fan.}
       \end{minipage}
    \end{figure}

Consider $n = 2$, i.e. the case of toric surfaces. In this case if $P_D$ has dimension~2 (i.e. it is full-dimensional), then $D$ is very ample. One can find the proof in~\cite[Corollary 2.2.18, Proposition 6.1.4]{C}.

Since $\Sigma$ is bilateral, $p_1$ and $p_2$ form a basis of $N$ and $p_3,\dotsc,p_k$ have the following form: \begin{equation}
    p_i = -\alpha_{i1}p_1-\alpha_{i2}p_2, \ i > 2, \ \alpha_{ij} \geq 0.
\end{equation} A fan $\Sigma$ is called \emph{wide} if there are $i_1$ and $i_2$ such that $\alpha_{i_1 1}>\alpha_{i_1 2}$ and $\alpha_{i_2 1}<\alpha_{i_2 2}$. This definition was introduced in~\cite{DZ}.

\begin{figure}[H]
    \begin{minipage}{0.48\textwidth}
     \centering    
    \begin{tikzpicture}[x=0.75pt,y=0.75pt,yscale=-1,xscale=1]

\draw    (167.47,819.01) -- (248.01,817.86) ;
\draw [shift={(250.01,817.83)}, rotate = 179.18] [color={rgb, 255:red, 0; green, 0; blue, 0 }  ][line width=0.75]    (10.93,-3.29) .. controls (6.95,-1.4) and (3.31,-0.3) .. (0,0) .. controls (3.31,0.3) and (6.95,1.4) .. (10.93,3.29)   ;
\draw    (167.47,819.01) -- (167.43,732.84) ;
\draw [shift={(167.42,730.84)}, rotate = 89.97] [color={rgb, 255:red, 0; green, 0; blue, 0 }  ][line width=0.75]    (10.93,-3.29) .. controls (6.95,-1.4) and (3.31,-0.3) .. (0,0) .. controls (3.31,0.3) and (6.95,1.4) .. (10.93,3.29)   ;
\draw    (167.47,819.01) -- (90.35,828.98) ;
\draw [shift={(88.37,829.23)}, rotate = 352.64] [color={rgb, 255:red, 0; green, 0; blue, 0 }  ][line width=0.75]    (10.93,-3.29) .. controls (6.95,-1.4) and (3.31,-0.3) .. (0,0) .. controls (3.31,0.3) and (6.95,1.4) .. (10.93,3.29)   ;
\draw    (167.47,819.01) -- (152.9,909.9) ;
\draw [shift={(152.59,911.88)}, rotate = 279.11] [color={rgb, 255:red, 0; green, 0; blue, 0 }  ][line width=0.75]    (10.93,-3.29) .. controls (6.95,-1.4) and (3.31,-0.3) .. (0,0) .. controls (3.31,0.3) and (6.95,1.4) .. (10.93,3.29)   ;
\draw    (77.33,819.36) -- (167.47,819.01) ;
\draw    (167.47,819.01) -- (167.14,902.52) ;
\draw    (167.47,819.01) -- (101.16,883.45) ;

\draw (242.26,824.24) node [anchor=north west][inner sep=0.75pt]  [font=\small]  {$p_{1}$};
\draw (175.94,732.41) node [anchor=north west][inner sep=0.75pt]  [font=\small]  {$p_{2}$};
\draw (89.08,836.21) node [anchor=north west][inner sep=0.75pt]  [font=\small]  {$p_{3}$};
\draw (130.75,899.48) node [anchor=north west][inner sep=0.75pt]  [font=\small]  {$p_{4}$};

\end{tikzpicture}
     \caption{A wide fan.}
   \end{minipage}
   \begin{minipage}{0.48\textwidth}
     \centering
    \begin{tikzpicture}[x=0.75pt,y=0.75pt,yscale=-1,xscale=1]

\draw    (359.15,818.72) -- (437.95,817.57) ;
\draw [shift={(439.95,817.54)}, rotate = 179.16] [color={rgb, 255:red, 0; green, 0; blue, 0 }  ][line width=0.75]    (10.93,-3.29) .. controls (6.95,-1.4) and (3.31,-0.3) .. (0,0) .. controls (3.31,0.3) and (6.95,1.4) .. (10.93,3.29)   ;
\draw    (359.15,818.72) -- (359.1,732.07) ;
\draw [shift={(359.1,730.07)}, rotate = 89.97] [color={rgb, 255:red, 0; green, 0; blue, 0 }  ][line width=0.75]    (10.93,-3.29) .. controls (6.95,-1.4) and (3.31,-0.3) .. (0,0) .. controls (3.31,0.3) and (6.95,1.4) .. (10.93,3.29)   ;
\draw    (359.15,818.72) -- (321.26,869.56) ;
\draw [shift={(320.06,871.16)}, rotate = 306.7] [color={rgb, 255:red, 0; green, 0; blue, 0 }  ][line width=0.75]    (10.93,-3.29) .. controls (6.95,-1.4) and (3.31,-0.3) .. (0,0) .. controls (3.31,0.3) and (6.95,1.4) .. (10.93,3.29)   ;
\draw    (359.15,818.72) -- (344.88,910.12) ;
\draw [shift={(344.57,912.1)}, rotate = 278.87] [color={rgb, 255:red, 0; green, 0; blue, 0 }  ][line width=0.75]    (10.93,-3.29) .. controls (6.95,-1.4) and (3.31,-0.3) .. (0,0) .. controls (3.31,0.3) and (6.95,1.4) .. (10.93,3.29)   ;
\draw    (270.9,819.07) -- (359.15,818.72) ;
\draw    (359.15,818.72) -- (358.82,902.69) ;
\draw    (359.15,818.72) -- (294.22,883.52) ;

\draw (435.3,821.1) node [anchor=north west][inner sep=0.75pt]  [font=\small]  {$p_{1}$};
\draw (368.35,730.69) node [anchor=north west][inner sep=0.75pt]  [font=\small]  {$p_{2}$};
\draw (314.73,876.4) node [anchor=north west][inner sep=0.75pt]  [font=\small]  {$p_{3}$};
\draw (323.99,904.43) node [anchor=north west][inner sep=0.75pt]  [font=\small]  {$p_{4}$};

\end{tikzpicture}
     \caption{A non-wide fan.}
   \end{minipage}
\end{figure}

We call a 2-dimensional lattice polytope $P$ \emph{elongated} if its normal fan is not wide.

\begin{theorem}\label{surfacesuntheorem}\cite[Theorem 3]{DZ} Let $X_\Sigma$ be a complete toric surface corresponding to a fan~$\Sigma$. Suppose $X_\Sigma$ admits an additive action. Then $X_\Sigma$ admits only one additive action if and only if $\Sigma$ is wide. Otherwise~$X_\Sigma$ admits exactly two non-equivalent additive actions.

\end{theorem}

Summarizing our reasoning above, we can formulate the following proposition. It describes the relationship between additive actions on complete toric surfaces and the corresponding very ample polytopes. 

\begin{proposition}There is a one-to-one correspondence between complete linearly normal toric surfaces~$X\subseteq~\BP^N$ admitting an additive action and polytopes~$P$ inscribed in a rectangle and containing~$N+1$ integer points. The fan of~$X$ is the normal fan of~$P$.
Moreover, there is a non-normalized additive action on~$X$ if and only if~$P$ is elongated.
\end{proposition}

\begin{remark}\label{derivationweight}
    Without loss of generality, for not wide fans we can assume that $\alpha_{i2}>\alpha_{i1}$ for~$i>2$. Consider the cone $\sigma$ generated by primitive vectors $p_1$ and $p_2$. The corresponding affine chart $U_\sigma=\Spec\BK[\sigma^{\vee}\cap M]$ is isomorphic to $\BA^2$. Choose coordinates $x, y$ on $U_\sigma$ such that\begin{equation*}
        (t_1,t_2)\cdot(x,y) = (t_1x,t_2y).
    \end{equation*} By Proposition~\ref{openorbit} we can assume that all additive actions have the same open orbit. From~\cite[Section 6]{DZ} it follows that we can assume that this open orbit coincides with~$U_\sigma$. Then the normalized additive action corresponds to the derivations $\partial/\partial x$ and $\partial/\partial y$ and the non-normalized one corresponds to the derivations $\partial/\partial x + x^d\partial/\partial y$ and~$\partial/\partial y$, where~$d =\lfloor\min_{i>2}\{\alpha_{i2}/\alpha_{i1}\}\rfloor$ (see \cite[Section 6 and Lemma 7]{DZ}).
\end{remark}

\begin{example}
    Consider the polytope $D_{1,n}$ that is the convex hull of the points $(0,0)$, $(n, 0)$ and $(0, 1)$. It defines a morphism $T\to\BP^{n+1}$ given by the formula\begin{equation}\label{wppmapping}
        \varphi_{D_{1,n}}(t_1,t_2)=[1:t_1:t_1^2:\dotsc:t_1^n:t_2], \ t_1, t_2 \in \BK^\times.
    \end{equation} The closure of the image gives us a toric projective surface with a fan that has primitive vectors $(1, 0)$, $(0, 1)$ and $(-1,-n)$.

    \begin{figure}[h]
    \begin{minipage}{0.495\textwidth}
     \centering 
     \begin{tikzpicture}[x=0.75pt,y=0.75pt,yscale=-1,xscale=1]

\draw    (49,1259.47) -- (49.41,1221.21) ;
\draw [shift={(49.43,1219.21)}, rotate = 90.61] [color={rgb, 255:red, 0; green, 0; blue, 0 }  ][line width=0.75]    (10.93,-3.29) .. controls (6.95,-1.4) and (3.31,-0.3) .. (0,0) .. controls (3.31,0.3) and (6.95,1.4) .. (10.93,3.29)   ;
\draw    (49,1259.47) -- (207.43,1258.23) ;
\draw [shift={(209.43,1258.21)}, rotate = 179.55] [color={rgb, 255:red, 0; green, 0; blue, 0 }  ][line width=0.75]    (10.93,-3.29) .. controls (6.95,-1.4) and (3.31,-0.3) .. (0,0) .. controls (3.31,0.3) and (6.95,1.4) .. (10.93,3.29)   ;
\draw  [dash pattern={on 0.84pt off 2.51pt}]  (49.21,1239.34) -- (189.43,1258.21) ;
\draw [shift={(189.43,1258.21)}, rotate = 7.67] [color={rgb, 255:red, 0; green, 0; blue, 0 }  ][fill={rgb, 255:red, 0; green, 0; blue, 0 }  ][line width=0.75]      (0, 0) circle [x radius= 1.34, y radius= 1.34]   ;
\draw [shift={(49.21,1239.34)}, rotate = 7.67] [color={rgb, 255:red, 0; green, 0; blue, 0 }  ][fill={rgb, 255:red, 0; green, 0; blue, 0 }  ][line width=0.75]      (0, 0) circle [x radius= 1.34, y radius= 1.34]   ;
\draw  [dash pattern={on 0.84pt off 2.51pt}]  (49.21,1239.34) -- (49,1259.47) ;
\draw [shift={(49,1259.47)}, rotate = 90.61] [color={rgb, 255:red, 0; green, 0; blue, 0 }  ][fill={rgb, 255:red, 0; green, 0; blue, 0 }  ][line width=0.75]      (0, 0) circle [x radius= 1.34, y radius= 1.34]   ;
\draw [shift={(49.21,1239.34)}, rotate = 90.61] [color={rgb, 255:red, 0; green, 0; blue, 0 }  ][fill={rgb, 255:red, 0; green, 0; blue, 0 }  ][line width=0.75]      (0, 0) circle [x radius= 1.34, y radius= 1.34]   ;
\draw  [dash pattern={on 0.84pt off 2.51pt}]  (69.21,1259.34) -- (49,1259.47) ;
\draw [shift={(49,1259.47)}, rotate = 179.64] [color={rgb, 255:red, 0; green, 0; blue, 0 }  ][fill={rgb, 255:red, 0; green, 0; blue, 0 }  ][line width=0.75]      (0, 0) circle [x radius= 1.34, y radius= 1.34]   ;
\draw [shift={(69.21,1259.34)}, rotate = 179.64] [color={rgb, 255:red, 0; green, 0; blue, 0 }  ][fill={rgb, 255:red, 0; green, 0; blue, 0 }  ][line width=0.75]      (0, 0) circle [x radius= 1.34, y radius= 1.34]   ;
\draw  [dash pattern={on 0.84pt off 2.51pt}]  (69.21,1259.34) -- (89.43,1259.21) ;
\draw [shift={(89.43,1259.21)}, rotate = 359.64] [color={rgb, 255:red, 0; green, 0; blue, 0 }  ][fill={rgb, 255:red, 0; green, 0; blue, 0 }  ][line width=0.75]      (0, 0) circle [x radius= 1.34, y radius= 1.34]   ;
\draw [shift={(69.21,1259.34)}, rotate = 359.64] [color={rgb, 255:red, 0; green, 0; blue, 0 }  ][fill={rgb, 255:red, 0; green, 0; blue, 0 }  ][line width=0.75]      (0, 0) circle [x radius= 1.34, y radius= 1.34]   ;
\draw  [dash pattern={on 0.84pt off 2.51pt}]  (89.43,1259.21) -- (109.43,1259.21) ;
\draw [shift={(109.43,1259.21)}, rotate = 0] [color={rgb, 255:red, 0; green, 0; blue, 0 }  ][fill={rgb, 255:red, 0; green, 0; blue, 0 }  ][line width=0.75]      (0, 0) circle [x radius= 1.34, y radius= 1.34]   ;
\draw [shift={(89.43,1259.21)}, rotate = 0] [color={rgb, 255:red, 0; green, 0; blue, 0 }  ][fill={rgb, 255:red, 0; green, 0; blue, 0 }  ][line width=0.75]      (0, 0) circle [x radius= 1.34, y radius= 1.34]   ;
\draw  [dash pattern={on 0.84pt off 2.51pt}]  (109.43,1259.21) -- (129.21,1258.84) ;
\draw [shift={(129.21,1258.84)}, rotate = 358.92] [color={rgb, 255:red, 0; green, 0; blue, 0 }  ][fill={rgb, 255:red, 0; green, 0; blue, 0 }  ][line width=0.75]      (0, 0) circle [x radius= 1.34, y radius= 1.34]   ;
\draw [shift={(109.43,1259.21)}, rotate = 358.92] [color={rgb, 255:red, 0; green, 0; blue, 0 }  ][fill={rgb, 255:red, 0; green, 0; blue, 0 }  ][line width=0.75]      (0, 0) circle [x radius= 1.34, y radius= 1.34]   ;
\draw  [dash pattern={on 0.84pt off 2.51pt}]  (149.43,1258.21) -- (129.21,1258.84) ;
\draw [shift={(129.21,1258.84)}, rotate = 178.22] [color={rgb, 255:red, 0; green, 0; blue, 0 }  ][fill={rgb, 255:red, 0; green, 0; blue, 0 }  ][line width=0.75]      (0, 0) circle [x radius= 1.34, y radius= 1.34]   ;
\draw [shift={(149.43,1258.21)}, rotate = 178.22] [color={rgb, 255:red, 0; green, 0; blue, 0 }  ][fill={rgb, 255:red, 0; green, 0; blue, 0 }  ][line width=0.75]      (0, 0) circle [x radius= 1.34, y radius= 1.34]   ;
\draw  [dash pattern={on 0.84pt off 2.51pt}]  (149.43,1258.21) -- (169.43,1258.21) ;
\draw [shift={(169.43,1258.21)}, rotate = 0] [color={rgb, 255:red, 0; green, 0; blue, 0 }  ][fill={rgb, 255:red, 0; green, 0; blue, 0 }  ][line width=0.75]      (0, 0) circle [x radius= 1.34, y radius= 1.34]   ;
\draw [shift={(149.43,1258.21)}, rotate = 0] [color={rgb, 255:red, 0; green, 0; blue, 0 }  ][fill={rgb, 255:red, 0; green, 0; blue, 0 }  ][line width=0.75]      (0, 0) circle [x radius= 1.34, y radius= 1.34]   ;
\draw  [dash pattern={on 0.84pt off 2.51pt}]  (169.43,1258.21) -- (189.43,1258.21) ;
\draw [shift={(189.43,1258.21)}, rotate = 0] [color={rgb, 255:red, 0; green, 0; blue, 0 }  ][fill={rgb, 255:red, 0; green, 0; blue, 0 }  ][line width=0.75]      (0, 0) circle [x radius= 1.34, y radius= 1.34]   ;
\draw [shift={(169.43,1258.21)}, rotate = 0] [color={rgb, 255:red, 0; green, 0; blue, 0 }  ][fill={rgb, 255:red, 0; green, 0; blue, 0 }  ][line width=0.75]      (0, 0) circle [x radius= 1.34, y radius= 1.34]   ;

\draw (34,1234.93) node [anchor=north west][inner sep=0.75pt]  [font=\scriptsize]  {$1$};
\draw (187,1263.93) node [anchor=north west][inner sep=0.75pt]  [font=\footnotesize]  {$n$};

\end{tikzpicture}
     \caption{The polytope $D_{1,n}$.}
   \end{minipage}
   \begin{minipage}{0.495\textwidth}
     \centering
     \begin{tikzpicture}[x=0.75pt,y=0.75pt,yscale=-1,xscale=1]

\draw    (360.43,1240.21) -- (360.43,1202.21) ;
\draw [shift={(360.43,1200.21)}, rotate = 90] [color={rgb, 255:red, 0; green, 0; blue, 0 }  ][line width=0.75]    (10.93,-3.29) .. controls (6.95,-1.4) and (3.31,-0.3) .. (0,0) .. controls (3.31,0.3) and (6.95,1.4) .. (10.93,3.29)   ;
\draw    (360.43,1240.21) -- (398.43,1240.21) ;
\draw [shift={(400.43,1240.21)}, rotate = 180] [color={rgb, 255:red, 0; green, 0; blue, 0 }  ][line width=0.75]    (10.93,-3.29) .. controls (6.95,-1.4) and (3.31,-0.3) .. (0,0) .. controls (3.31,0.3) and (6.95,1.4) .. (10.93,3.29)   ;
\draw    (360.43,1240.21) -- (340.91,1318.27) ;
\draw [shift={(340.43,1320.21)}, rotate = 284.04] [color={rgb, 255:red, 0; green, 0; blue, 0 }  ][line width=0.75]    (10.93,-3.29) .. controls (6.95,-1.4) and (3.31,-0.3) .. (0,0) .. controls (3.31,0.3) and (6.95,1.4) .. (10.93,3.29)   ;

\draw (393,1248.93) node [anchor=north west][inner sep=0.75pt]  [font=\scriptsize]  {$( 1,0)$};
\draw (368,1197.93) node [anchor=north west][inner sep=0.75pt]  [font=\scriptsize]  {$( 0,1)$};
\draw (350,1310.93) node [anchor=north west][inner sep=0.75pt]  [font=\scriptsize]  {$( -1,-n)$};

\end{tikzpicture}
     \caption{The fan of $\BP(1,1,n)$.}
   \end{minipage}
\end{figure}
    
    This is the weighted projective plane $\BP(1,1,n)$. Its fan is not wide, so it has a second non-normalized additive action induced by the derivations $\delta_1 =\partial/\partial x + x^n\partial/\partial y$ and $\delta_2 = \partial/\partial y$ (see Remark~\ref{derivationweight}). The first additive action corresponds to the $S$-pair with the algebra~$W_n=\BK[x,y]/(x^{n+1},xy,y^2)$ and the subspace $U_n=\langle x, y\rangle$ by Theorem~\ref{algebranorm}. The second one corresponds to the $S$-pair\begin{equation}
        W_n' = \BK[\delta_1|_{W_n}, \delta_2|_{W_n}], \ U_n' = \langle\delta_1|_{W_n},\delta_2|_{W_n}\rangle
    \end{equation} by Corollary~\ref{secondalgebra}.

    It is easy to see that on $W_n$ the following relations hold:\begin{equation*}
        \delta_2^2=0, \ \delta_1\delta_2=0, \ \delta_1^{n+1} = n!\delta_2,
    \end{equation*} so we have \begin{equation}
        W_n' \simeq \BK[t,s]/(t^{n+1}-n!s, ts, s^2)\simeq\BK[t,s]/(t^{n+1}-s,ts,s^2)\simeq{\BK[t]/(t^{n+2})},
    \end{equation} and under these isomorphisms $U_n'$ maps to the subspace $\langle t, t^{n+1}\rangle$. This yields an $S$-pair that corresponds to the non-normalized additive action:\begin{equation}
        (\BK[t]/(t^{n+2}), \ \langle t, t^{n+1}\rangle).
    \end{equation}
\end{example}

\begin{example}\label{hirzalgebraslow}
    Now consider the polytope $D_{2,n-1}$ that is the convex hull of the points $(0,0)$, $(n, 0)$, $(0, 1)$ and $(1,1)$, $n\geq2$. It defines a morphism $T\to\BP^{n+2}$:\begin{equation}\label{hirzmapping}
        \varphi_{D_{2,n-1}}(t_1,t_2)=[1:t_1:t_1^2:\dotsc:t_1^n:t_1t_2:t_2], \ t_1,t_2\in\BK^\times.
    \end{equation}
    The closure gives us a projective toric surface with a fan with primitive vectors $(1, 0)$, $(0, 1)$, $(-1,-n+1)$ and $(0, -1)$. It is the fan of the Hirzebruch surface $\BF_{n-1}$.

     \begin{figure}[h]
        \begin{minipage}{0.52\textwidth}
         \centering 
         \begin{tikzpicture}[x=0.75pt,y=0.75pt,yscale=-1,xscale=1]

\draw    (41,1260.47) -- (41.41,1222.21) ;
\draw [shift={(41.43,1220.21)}, rotate = 90.61] [color={rgb, 255:red, 0; green, 0; blue, 0 }  ][line width=0.75]    (10.93,-3.29) .. controls (6.95,-1.4) and (3.31,-0.3) .. (0,0) .. controls (3.31,0.3) and (6.95,1.4) .. (10.93,3.29)   ;
\draw    (41,1260.47) -- (199.43,1259.23) ;
\draw [shift={(201.43,1259.21)}, rotate = 179.55] [color={rgb, 255:red, 0; green, 0; blue, 0 }  ][line width=0.75]    (10.93,-3.29) .. controls (6.95,-1.4) and (3.31,-0.3) .. (0,0) .. controls (3.31,0.3) and (6.95,1.4) .. (10.93,3.29)   ;
\draw  [dash pattern={on 0.84pt off 2.51pt}]  (41.21,1240.34) -- (41,1260.47) ;
\draw [shift={(41,1260.47)}, rotate = 90.61] [color={rgb, 255:red, 0; green, 0; blue, 0 }  ][fill={rgb, 255:red, 0; green, 0; blue, 0 }  ][line width=0.75]      (0, 0) circle [x radius= 1.34, y radius= 1.34]   ;
\draw [shift={(41.21,1240.34)}, rotate = 90.61] [color={rgb, 255:red, 0; green, 0; blue, 0 }  ][fill={rgb, 255:red, 0; green, 0; blue, 0 }  ][line width=0.75]      (0, 0) circle [x radius= 1.34, y radius= 1.34]   ;
\draw  [dash pattern={on 0.84pt off 2.51pt}]  (61.21,1260.34) -- (41,1260.47) ;
\draw [shift={(41,1260.47)}, rotate = 179.64] [color={rgb, 255:red, 0; green, 0; blue, 0 }  ][fill={rgb, 255:red, 0; green, 0; blue, 0 }  ][line width=0.75]      (0, 0) circle [x radius= 1.34, y radius= 1.34]   ;
\draw [shift={(61.21,1260.34)}, rotate = 179.64] [color={rgb, 255:red, 0; green, 0; blue, 0 }  ][fill={rgb, 255:red, 0; green, 0; blue, 0 }  ][line width=0.75]      (0, 0) circle [x radius= 1.34, y radius= 1.34]   ;
\draw  [dash pattern={on 0.84pt off 2.51pt}]  (61.21,1260.34) -- (81.43,1260.21) ;
\draw [shift={(81.43,1260.21)}, rotate = 359.64] [color={rgb, 255:red, 0; green, 0; blue, 0 }  ][fill={rgb, 255:red, 0; green, 0; blue, 0 }  ][line width=0.75]      (0, 0) circle [x radius= 1.34, y radius= 1.34]   ;
\draw [shift={(61.21,1260.34)}, rotate = 359.64] [color={rgb, 255:red, 0; green, 0; blue, 0 }  ][fill={rgb, 255:red, 0; green, 0; blue, 0 }  ][line width=0.75]      (0, 0) circle [x radius= 1.34, y radius= 1.34]   ;
\draw  [dash pattern={on 0.84pt off 2.51pt}]  (81.43,1260.21) -- (101.43,1260.21) ;
\draw [shift={(101.43,1260.21)}, rotate = 0] [color={rgb, 255:red, 0; green, 0; blue, 0 }  ][fill={rgb, 255:red, 0; green, 0; blue, 0 }  ][line width=0.75]      (0, 0) circle [x radius= 1.34, y radius= 1.34]   ;
\draw [shift={(81.43,1260.21)}, rotate = 0] [color={rgb, 255:red, 0; green, 0; blue, 0 }  ][fill={rgb, 255:red, 0; green, 0; blue, 0 }  ][line width=0.75]      (0, 0) circle [x radius= 1.34, y radius= 1.34]   ;
\draw  [dash pattern={on 0.84pt off 2.51pt}]  (101.43,1260.21) -- (121.21,1259.84) ;
\draw [shift={(121.21,1259.84)}, rotate = 358.92] [color={rgb, 255:red, 0; green, 0; blue, 0 }  ][fill={rgb, 255:red, 0; green, 0; blue, 0 }  ][line width=0.75]      (0, 0) circle [x radius= 1.34, y radius= 1.34]   ;
\draw [shift={(101.43,1260.21)}, rotate = 358.92] [color={rgb, 255:red, 0; green, 0; blue, 0 }  ][fill={rgb, 255:red, 0; green, 0; blue, 0 }  ][line width=0.75]      (0, 0) circle [x radius= 1.34, y radius= 1.34]   ;
\draw  [dash pattern={on 0.84pt off 2.51pt}]  (141.43,1259.21) -- (121.21,1259.84) ;
\draw [shift={(121.21,1259.84)}, rotate = 178.22] [color={rgb, 255:red, 0; green, 0; blue, 0 }  ][fill={rgb, 255:red, 0; green, 0; blue, 0 }  ][line width=0.75]      (0, 0) circle [x radius= 1.34, y radius= 1.34]   ;
\draw [shift={(141.43,1259.21)}, rotate = 178.22] [color={rgb, 255:red, 0; green, 0; blue, 0 }  ][fill={rgb, 255:red, 0; green, 0; blue, 0 }  ][line width=0.75]      (0, 0) circle [x radius= 1.34, y radius= 1.34]   ;
\draw  [dash pattern={on 0.84pt off 2.51pt}]  (141.43,1259.21) -- (161.43,1259.21) ;
\draw [shift={(161.43,1259.21)}, rotate = 0] [color={rgb, 255:red, 0; green, 0; blue, 0 }  ][fill={rgb, 255:red, 0; green, 0; blue, 0 }  ][line width=0.75]      (0, 0) circle [x radius= 1.34, y radius= 1.34]   ;
\draw [shift={(141.43,1259.21)}, rotate = 0] [color={rgb, 255:red, 0; green, 0; blue, 0 }  ][fill={rgb, 255:red, 0; green, 0; blue, 0 }  ][line width=0.75]      (0, 0) circle [x radius= 1.34, y radius= 1.34]   ;
\draw  [dash pattern={on 0.84pt off 2.51pt}]  (161.43,1259.21) -- (181.43,1259.21) ;
\draw [shift={(181.43,1259.21)}, rotate = 0] [color={rgb, 255:red, 0; green, 0; blue, 0 }  ][fill={rgb, 255:red, 0; green, 0; blue, 0 }  ][line width=0.75]      (0, 0) circle [x radius= 1.34, y radius= 1.34]   ;
\draw [shift={(161.43,1259.21)}, rotate = 0] [color={rgb, 255:red, 0; green, 0; blue, 0 }  ][fill={rgb, 255:red, 0; green, 0; blue, 0 }  ][line width=0.75]      (0, 0) circle [x radius= 1.34, y radius= 1.34]   ;
\draw  [dash pattern={on 0.84pt off 2.51pt}]  (61.43,1240.21) -- (41.21,1240.34) ;
\draw [shift={(41.21,1240.34)}, rotate = 179.64] [color={rgb, 255:red, 0; green, 0; blue, 0 }  ][fill={rgb, 255:red, 0; green, 0; blue, 0 }  ][line width=0.75]      (0, 0) circle [x radius= 1.34, y radius= 1.34]   ;
\draw [shift={(61.43,1240.21)}, rotate = 179.64] [color={rgb, 255:red, 0; green, 0; blue, 0 }  ][fill={rgb, 255:red, 0; green, 0; blue, 0 }  ][line width=0.75]      (0, 0) circle [x radius= 1.34, y radius= 1.34]   ;
\draw  [dash pattern={on 0.84pt off 2.51pt}]  (181.43,1259.21) -- (61.43,1240.21) ;
\draw [shift={(61.43,1240.21)}, rotate = 189] [color={rgb, 255:red, 0; green, 0; blue, 0 }  ][fill={rgb, 255:red, 0; green, 0; blue, 0 }  ][line width=0.75]      (0, 0) circle [x radius= 1.34, y radius= 1.34]   ;
\draw [shift={(181.43,1259.21)}, rotate = 189] [color={rgb, 255:red, 0; green, 0; blue, 0 }  ][fill={rgb, 255:red, 0; green, 0; blue, 0 }  ][line width=0.75]      (0, 0) circle [x radius= 1.34, y radius= 1.34]   ;
\draw  [dash pattern={on 0.84pt off 2.51pt}]  (61.21,1260.34) -- (61.43,1240.21) ;
\draw [shift={(61.43,1240.21)}, rotate = 270.61] [color={rgb, 255:red, 0; green, 0; blue, 0 }  ][fill={rgb, 255:red, 0; green, 0; blue, 0 }  ][line width=0.75]      (0, 0) circle [x radius= 1.34, y radius= 1.34]   ;
\draw [shift={(61.21,1260.34)}, rotate = 270.61] [color={rgb, 255:red, 0; green, 0; blue, 0 }  ][fill={rgb, 255:red, 0; green, 0; blue, 0 }  ][line width=0.75]      (0, 0) circle [x radius= 1.34, y radius= 1.34]   ;

\draw (26,1235.93) node [anchor=north west][inner sep=0.75pt]  [font=\scriptsize]  {$1$};
\draw (179,1264.93) node [anchor=north west][inner sep=0.75pt]  [font=\footnotesize]  {$n$};
\draw (60,1269.93) node [anchor=north west][inner sep=0.75pt]  [font=\scriptsize]  {$1$};

\end{tikzpicture}
         \caption{The polytope $D_{2,n-1}$.}
       \end{minipage}
       \begin{minipage}{0.47\textwidth}
         \centering
         \begin{tikzpicture}[x=0.75pt,y=0.75pt,yscale=-1,xscale=1]

\draw    (360.43,1240.21) -- (360.43,1202.21) ;
\draw [shift={(360.43,1200.21)}, rotate = 90] [color={rgb, 255:red, 0; green, 0; blue, 0 }  ][line width=0.75]    (10.93,-3.29) .. controls (6.95,-1.4) and (3.31,-0.3) .. (0,0) .. controls (3.31,0.3) and (6.95,1.4) .. (10.93,3.29)   ;
\draw    (360.43,1240.21) -- (398.43,1240.21) ;
\draw [shift={(400.43,1240.21)}, rotate = 180] [color={rgb, 255:red, 0; green, 0; blue, 0 }  ][line width=0.75]    (10.93,-3.29) .. controls (6.95,-1.4) and (3.31,-0.3) .. (0,0) .. controls (3.31,0.3) and (6.95,1.4) .. (10.93,3.29)   ;
\draw    (360.43,1240.21) -- (341.07,1297.72) ;
\draw [shift={(340.43,1299.62)}, rotate = 288.61] [color={rgb, 255:red, 0; green, 0; blue, 0 }  ][line width=0.75]    (10.93,-3.29) .. controls (6.95,-1.4) and (3.31,-0.3) .. (0,0) .. controls (3.31,0.3) and (6.95,1.4) .. (10.93,3.29)   ;
\draw    (360.43,1240.21) -- (360.43,1277.62) ;
\draw [shift={(360.43,1279.62)}, rotate = 270] [color={rgb, 255:red, 0; green, 0; blue, 0 }  ][line width=0.75]    (10.93,-3.29) .. controls (6.95,-1.4) and (3.31,-0.3) .. (0,0) .. controls (3.31,0.3) and (6.95,1.4) .. (10.93,3.29)   ;

\draw (393,1248.93) node [anchor=north west][inner sep=0.75pt]  [font=\scriptsize]  {$( 1,0)$};
\draw (368,1197.93) node [anchor=north west][inner sep=0.75pt]  [font=\scriptsize]  {$( 0,1)$};
\draw (264,1293.93) node [anchor=north west][inner sep=0.75pt]  [font=\scriptsize]  {$( -1,-n+1)$};
\draw (366,1284.93) node [anchor=north west][inner sep=0.75pt]  [font=\scriptsize]  {$( 0,-1)$};

\end{tikzpicture}
         \caption{The fan of $\BF_{n-1}$.}
       \end{minipage}
    \end{figure}

    The fan is not wide, so again we have the second non-normalized additive action induced by derivations $\delta_1 = \partial/\partial x + x^{n-1}\partial/\partial y$ and $\delta_2 = \partial/\partial y$ (see Remark~\ref{derivationweight}). By Theorem~\ref{algebranorm}, the normalized action corresponds to the $S$-pair consisting of the algebra $H_{n-1}=\BK[x,y]/(x^{n+1},y^2,x^2y)$ and the subspace $V_{n-1}=\langle x,y\rangle$. By Corollary~\ref{secondalgebra}, the non-normalized additive action corresponds to the following $S$-pair:\begin{equation}
        H_{n-1}'= \BK[\delta_1|_{H_{n-1}},\delta_2|_{H_{n-1}}], \ V_{n-1}' = \langle\delta_1|_{H_{n-1}},\delta_2|_{H_{n-1}}\rangle.
    \end{equation}
    Now we want to understand relations on $\delta_1$ and $\delta_2$ restricted to $H_{n-1}$. Further in the example, by $\delta_1$ and $\delta_2$ we mean their restrictions to the algebra $H_{n-1}.$  Firstly, a trivial induction by $k$ gives us\begin{equation*}
        \delta_1^k = \Big(\frac{\partial}{\partial x}\Big)^k+\frac{(n-1)!}{(n-k)!}x^{n-k}\frac{\partial}{\partial y} + k\frac{(n-1)!}{(n-k+1)!}x^{n-k+1}\frac{\partial^2}{\partial x\partial y}, \ k\geq 2,
    \end{equation*} which yields\begin{gather*}
        \delta_1^{n+2} = 0, \\ \delta_1^{n+1} = (n+1)(n-1)!\frac{\partial^2}{\partial x\partial y} = (n+1)(n-1)!\delta_1\delta_2 = \frac{(n+1)!}{n}\delta_1\delta_2.
    \end{gather*}
    Now we can represent $H_{n-1}'$ in the following form:\begin{equation}
        H_{n-1}'\simeq\BK[u,v]/(u^{n+2}, u^{n+1}-\frac{(n+1)!}{n}uv, v^2)\simeq\BK[u,v]/(u^{n+2}, u^{n+1}-uv,v^2),
    \end{equation} and under these isomorphisms $V_{n-1}'$ maps to $\langle u, v\rangle$. Now consider the next substitution:\begin{equation*}
        t = u, \ s = v - u^n,
    \end{equation*}we have \begin{equation*}
        t^{n+2}=u^{n+2} = 0, \ ts = uv -u^{n+1} = 0, \ s^2 = v^2-2vu^{n}+u^{2n}=0.
    \end{equation*}Consequently, $H_{n-1}'$ is isomorphic to the algebra \begin{equation*}
        \BK[t,s]/(t^{n+2},ts,s^2),
    \end{equation*} and under all these isomorphisms $V_{n-1}'$ maps to $\langle t,s+t^n\rangle$. Hence the $S$-pair corresponding to the non-normalized additive action on $\BF_{n-1}$ under this embedding in $\BP^{n+2}$ has the following form:\begin{equation}
        (\BK[t,s]/(t^{n+2},ts,s^2), \langle t,s+t^n\rangle).
    \end{equation}
\end{example}

\begin{example}\label{projplane}
    Now consider the polytope $D_3$ that is the convex hull of the points $(0,0)$, $(2,0)$ and~$(0,2)$. It defines the morphism\begin{equation}\label{projplanemapping}
        \varphi_{D_3}(t_1,t_2)=[1:t_1:t_1^2:t_2:t_2^2:t_1t_2], \ t_1,t_2\in\BK^\times.
    \end{equation}It actually corresponds to the embedding of $\BP^2$ in $\BP^5$ with $\CO(2)$.

     \begin{figure}[h]
        \begin{minipage}{0.48\textwidth}
         \centering
         \begin{tikzpicture}[x=0.75pt,y=0.75pt,yscale=-1,xscale=1]

\draw    (41,1260.47) -- (40.45,1202.62) ;
\draw [shift={(40.43,1200.62)}, rotate = 89.45] [color={rgb, 255:red, 0; green, 0; blue, 0 }  ][line width=0.75]    (10.93,-3.29) .. controls (6.95,-1.4) and (3.31,-0.3) .. (0,0) .. controls (3.31,0.3) and (6.95,1.4) .. (10.93,3.29)   ;
\draw    (41,1260.47) -- (98.43,1259.64) ;
\draw [shift={(100.43,1259.62)}, rotate = 179.18] [color={rgb, 255:red, 0; green, 0; blue, 0 }  ][line width=0.75]    (10.93,-3.29) .. controls (6.95,-1.4) and (3.31,-0.3) .. (0,0) .. controls (3.31,0.3) and (6.95,1.4) .. (10.93,3.29)   ;
\draw  [dash pattern={on 0.84pt off 2.51pt}]  (41.21,1240.34) -- (41,1260.47) ;
\draw [shift={(41,1260.47)}, rotate = 90.61] [color={rgb, 255:red, 0; green, 0; blue, 0 }  ][fill={rgb, 255:red, 0; green, 0; blue, 0 }  ][line width=0.75]      (0, 0) circle [x radius= 1.34, y radius= 1.34]   ;
\draw [shift={(41.21,1240.34)}, rotate = 90.61] [color={rgb, 255:red, 0; green, 0; blue, 0 }  ][fill={rgb, 255:red, 0; green, 0; blue, 0 }  ][line width=0.75]      (0, 0) circle [x radius= 1.34, y radius= 1.34]   ;
\draw  [dash pattern={on 0.84pt off 2.51pt}]  (61.21,1260.34) -- (41,1260.47) ;
\draw [shift={(41,1260.47)}, rotate = 179.64] [color={rgb, 255:red, 0; green, 0; blue, 0 }  ][fill={rgb, 255:red, 0; green, 0; blue, 0 }  ][line width=0.75]      (0, 0) circle [x radius= 1.34, y radius= 1.34]   ;
\draw [shift={(61.21,1260.34)}, rotate = 179.64] [color={rgb, 255:red, 0; green, 0; blue, 0 }  ][fill={rgb, 255:red, 0; green, 0; blue, 0 }  ][line width=0.75]      (0, 0) circle [x radius= 1.34, y radius= 1.34]   ;
\draw  [dash pattern={on 0.84pt off 2.51pt}]  (61.21,1260.34) -- (81.43,1260.21) ;
\draw [shift={(81.43,1260.21)}, rotate = 359.64] [color={rgb, 255:red, 0; green, 0; blue, 0 }  ][fill={rgb, 255:red, 0; green, 0; blue, 0 }  ][line width=0.75]      (0, 0) circle [x radius= 1.34, y radius= 1.34]   ;
\draw [shift={(61.21,1260.34)}, rotate = 359.64] [color={rgb, 255:red, 0; green, 0; blue, 0 }  ][fill={rgb, 255:red, 0; green, 0; blue, 0 }  ][line width=0.75]      (0, 0) circle [x radius= 1.34, y radius= 1.34]   ;
\draw  [dash pattern={on 0.84pt off 2.51pt}]  (81.43,1260.21) -- (61.43,1240.21) ;
\draw [shift={(61.43,1240.21)}, rotate = 225] [color={rgb, 255:red, 0; green, 0; blue, 0 }  ][fill={rgb, 255:red, 0; green, 0; blue, 0 }  ][line width=0.75]      (0, 0) circle [x radius= 1.34, y radius= 1.34]   ;
\draw [shift={(81.43,1260.21)}, rotate = 225] [color={rgb, 255:red, 0; green, 0; blue, 0 }  ][fill={rgb, 255:red, 0; green, 0; blue, 0 }  ][line width=0.75]      (0, 0) circle [x radius= 1.34, y radius= 1.34]   ;
\draw  [dash pattern={on 0.84pt off 2.51pt}]  (61.43,1240.21) -- (41.21,1240.34) ;
\draw [shift={(41.21,1240.34)}, rotate = 179.64] [color={rgb, 255:red, 0; green, 0; blue, 0 }  ][fill={rgb, 255:red, 0; green, 0; blue, 0 }  ][line width=0.75]      (0, 0) circle [x radius= 1.34, y radius= 1.34]   ;
\draw [shift={(61.43,1240.21)}, rotate = 179.64] [color={rgb, 255:red, 0; green, 0; blue, 0 }  ][fill={rgb, 255:red, 0; green, 0; blue, 0 }  ][line width=0.75]      (0, 0) circle [x radius= 1.34, y radius= 1.34]   ;
\draw  [dash pattern={on 0.84pt off 2.51pt}]  (40.43,1219.62) -- (61.43,1240.21) ;
\draw [shift={(61.43,1240.21)}, rotate = 44.45] [color={rgb, 255:red, 0; green, 0; blue, 0 }  ][fill={rgb, 255:red, 0; green, 0; blue, 0 }  ][line width=0.75]      (0, 0) circle [x radius= 1.34, y radius= 1.34]   ;
\draw [shift={(40.43,1219.62)}, rotate = 44.45] [color={rgb, 255:red, 0; green, 0; blue, 0 }  ][fill={rgb, 255:red, 0; green, 0; blue, 0 }  ][line width=0.75]      (0, 0) circle [x radius= 1.34, y radius= 1.34]   ;
\draw  [dash pattern={on 0.84pt off 2.51pt}]  (61.21,1260.34) -- (61.43,1240.21) ;
\draw [shift={(61.43,1240.21)}, rotate = 270.61] [color={rgb, 255:red, 0; green, 0; blue, 0 }  ][fill={rgb, 255:red, 0; green, 0; blue, 0 }  ][line width=0.75]      (0, 0) circle [x radius= 1.34, y radius= 1.34]   ;
\draw [shift={(61.21,1260.34)}, rotate = 270.61] [color={rgb, 255:red, 0; green, 0; blue, 0 }  ][fill={rgb, 255:red, 0; green, 0; blue, 0 }  ][line width=0.75]      (0, 0) circle [x radius= 1.34, y radius= 1.34]   ;

\draw (26,1220.93) node [anchor=north west][inner sep=0.75pt]  [font=\scriptsize]  {$2$};
\draw (81,1268.93) node [anchor=north west][inner sep=0.75pt]  [font=\scriptsize]  {$2$};

\end{tikzpicture}
         \caption{The polytope $D_3$.}
       \end{minipage}
       \begin{minipage}{0.48\textwidth}
         \centering
       \begin{tikzpicture}[x=0.75pt,y=0.75pt,yscale=-1,xscale=1]

\draw    (360.43,1240.21) -- (360.43,1202.21) ;
\draw [shift={(360.43,1200.21)}, rotate = 90] [color={rgb, 255:red, 0; green, 0; blue, 0 }  ][line width=0.75]    (10.93,-3.29) .. controls (6.95,-1.4) and (3.31,-0.3) .. (0,0) .. controls (3.31,0.3) and (6.95,1.4) .. (10.93,3.29)   ;
\draw    (360.43,1240.21) -- (398.43,1240.21) ;
\draw [shift={(400.43,1240.21)}, rotate = 180] [color={rgb, 255:red, 0; green, 0; blue, 0 }  ][line width=0.75]    (10.93,-3.29) .. controls (6.95,-1.4) and (3.31,-0.3) .. (0,0) .. controls (3.31,0.3) and (6.95,1.4) .. (10.93,3.29)   ;
\draw    (360.43,1240.21) -- (332.83,1268.19) ;
\draw [shift={(331.43,1269.62)}, rotate = 314.61] [color={rgb, 255:red, 0; green, 0; blue, 0 }  ][line width=0.75]    (10.93,-3.29) .. controls (6.95,-1.4) and (3.31,-0.3) .. (0,0) .. controls (3.31,0.3) and (6.95,1.4) .. (10.93,3.29)   ;

\draw (393,1248.93) node [anchor=north west][inner sep=0.75pt]  [font=\scriptsize]  {$( 1,0)$};
\draw (368,1197.93) node [anchor=north west][inner sep=0.75pt]  [font=\scriptsize]  {$( 0,1)$};
\draw (338,1274.93) node [anchor=north west][inner sep=0.75pt]  [font=\scriptsize]  {$( -1,-1)$};

\end{tikzpicture}
         \caption{The fan of $\BP^2$.}
       \end{minipage}
    \end{figure}

    The normalized additive action corresponds to the $S$-pair\begin{equation*}
        A = \BK[x,y]/(x^3,y^3,x^2y,xy^2), \ U = \langle x, y \rangle
    \end{equation*}by Theorem~\ref{algebranorm}. The non-normalized additive action is induced by the derivations~$\delta_1=\partial/\partial x+x\partial/\partial y$, $\delta_2 = \partial/\partial y$ by Remark~\ref{derivationweight}. By Corollary~\ref{secondalgebra} the non-normalized additive action corresponds to the $S$-pair\begin{equation*}
        A' = \BK[\delta_1|_A,\delta_2|_A], \ U' = \langle\delta_1|_A,\delta_2|_A\rangle.
    \end{equation*} One can check that on $A$ the following equations hold: \begin{equation*}
        \delta_1^3=3\delta_1\delta_2, \ \delta_1^4=3\delta_2^2, \ \delta_1^5 = 0, \ \delta_2^3 = 0.
    \end{equation*} The substitution\begin{equation*}
        \widetilde{\delta_2} = \frac{3}{\sqrt{2}}(\delta_2-\frac{\delta_1^2}{3})
    \end{equation*} gives new relations\begin{equation*}
        \delta_1\widetilde{\delta_2}=0,  \ \delta_1^4 = \widetilde{\delta_2}^2.
    \end{equation*} Hence we get that $S$-pair $(A',U')$ is isomorphic to the following one:\begin{equation}
        (\BK[t,s]/(ts,t^4-s^2), \langle t, s\sqrt2+t^2\rangle).
    \end{equation}

\end{example}

It is easy to check that for $N\leq5$ there are no other elongated inscribed in rectangle polytopes except $D_{1,n}$ ($n\leq4$), $D_{2,n}$ ($n\leq2$) and $D_3$. Hence we can formulate the following theorem.

\begin{theorem}
    Let $X\subseteq\BP^N$ be a linearly normal toric surface admitting a non-normalized additive action. Let $(A,U)$ be the corresponding $S$-pair and $D$ be the corresponding elongated polytope. Then for $N\leq5$ all options are listed in Table~1.
\end{theorem}

\begin{table}[h]
\begin{tabular}{ |Sc|Sc|Sc|Sc|Sc| } 
\hline
$N$  & $A$ & $U$ & $X$ & $D$\\
\hline
\multirow{1}{1em}{$2$} &  $\BK[x]/(x^3)$ & $\langle x,x^{2}\rangle$ & $\BP^2$ & $D_{1,1}$ \\
\hline
\multirow{1}{1em}{$3$} &  $\BK[x]/(x^4)$ & $\langle x,x^{3}\rangle$ & $\BP(1,1,2)$ & $D_{1,2}$ \\
\hline
\multirow{2}{1em}{$4$} & $ \BK[x]/(x^5)$ & $\langle x,x^4\rangle$ & $\BP(1,1,3)$ & $D_{1,3}$  \\ 
&  $\BK[x,y]/(x^4,xy,y^2)$ & $\langle x, y+x^2\rangle$ & $\BF_1$ & $D_{2,1}$\\ 
\hline
\multirow{3}{1em}{$5$} & $\BK[x]/(x^6)$ & $\langle x,x^5\rangle$ & $\BP(1,1,4)$ & $D_{1,4}$\\
&  $\BK[x,y]/(x^5,xy,y^2)$ & $\langle x, y+x^3\rangle$ & $\BF_2$ & $D_{2,2}$ \\
&  $\BK[x,y]/(x^4-y^2)$ & $\langle x, y\sqrt{2}+x^2\rangle$ & $\BP^2$ & $D_3$ \\
\hline
\end{tabular}
\caption{
    $S$-pairs of linearly normal toric surfaces with a non-normalized additive action.}
\end{table}

For any lattice polytope $D$ inscribed in a rectangle there is a monomial algebra $A_D$ consisting of monomials corresponding to integer points of this polytope. As we proved in Theorem~\ref{algebranorm} the algebra $A_D$ together with the subspace spanned by the variables correspond to the normalized additive action on $X = \overline{\Im\varphi_D}\subseteq\BP^N$. In the case of toric surfaces $A_D$ has to be~2-generated. We present the list of such algebras $A$ of dimension $\leq6$, i.e. for $N\leq5$.

\begin{center}
\begin{table}[h]
\begin{tabular}{ |Sc|Sc|Sc| } 
\hline
$N$  & $A$ & $X$ \\
\hline
\multirow{1}{1em}{$2$} &  $ \BK[x,y]/(x^2,xy,y^2)$ & $\BP^2$ \\
\hline
\multirow{2}{1em}{$3$} & $\BK[x,y]/(x^3,xy,y^2)$ & $\BP(1,1,2)$ \\
& $\BK[x,y]/(x^2,y^2)$ & $\BP^1\times\BP^1$ \\
\hline
\multirow{2}{1em}{$4$} &  $ \BK[x,y]/(x^4,xy,y^2)$ & $\BP(1,1,3)$ \\ 
&   $\BK[x,y]/(x^3,x^2y,y^2)$ & $\BF_1$ \\ 
\hline
\multirow{4}{1em}{$5$} & $\BK[x,y]/(x^3,y^2)$ & $\BP^1\times\BP^1$ \\
& $\BK[x,y]/(x^5,xy,y^2)$ & $\BP(1,1,4)$ \\
& $\BK[x,y]/(x^4,y^2,x^2y)$ & $\BF_2$ \\
&  $\BK[x,y]/(x^3,y^3,x^2y,xy^2)$ & $\BP^2$ \\
\hline
\end{tabular}
\caption{
    Monomial 2-generated local algebras of dimension $\leq 6$ corresponding to lattice polytopes inscribed in a rectangle.}
\end{table}
\end{center}

\section*{Acknowledgments}

The author is grateful to Anton Shafarevich for his direct involvement in editing the work and for providing references on additive actions on toric varieties. Special thanks go to Ivan~Arzhantsev for his many valuable suggestions and corrections, which greatly improved the text.

\end{document}